\documentclass[12pt]{amsart}
\usepackage{amsfonts,amssymb,amscd,amsmath,enumerate,verbatim} 
\def\NZQ{\mathbb}               

\def\QQ{{\NZQ Q}}
\def\ZZ{{\NZQ Z}}
\def\RR{{\NZQ R}}

\def\frk{\mathfrak}               

\def\Phi{{\frk N}}
\def\opn#1#2{\def#1{\operatorname{#2}}} 
\opn\chara{char} 
\opn\length{\ell} 
\opn\pd{pd} 
\opn\rk{rk}
\opn\projdim{proj\,dim} 
\opn\injdim{inj\,dim} 
\opn\rank{rank}
\opn\depth{depth} 
\opn\grade{grade} 
\opn\height{height}
\opn\embdim{emb\,dim} 
\opn\codim{codim}

\opn\Tr{Tr} 
\opn\bigrank{big\,rank}
\opn\superheight{superheight}
\opn\lcm{lcm}
\opn\trdeg{tr\,deg}
\opn\reg{reg} 
\opn\lreg{lreg} 
\opn\ini{in} 
\opn\lpd{lpd}
\opn\size{size}
\opn\mult{mult}
\opn\dist{dist}
\opn\cone{cone}
\opn\lex{lex}
\opn\rev{rev}
\opn\div{div} \opn\Div{Div} \opn\cl{cl} \opn\Cl{Cl}
\opn\Spec{Spec} \opn\Supp{Supp} \opn\supp{supp} \opn\Sing{Sing}
\opn\Ass{Ass} \opn\Min{Min}
\opn\Ann{Ann} \opn\Rad{Rad} \opn\Soc{Soc}
\opn\Syz{Syz} \opn\Im{Im} \opn\Ker{Ker} \opn\Coker{Coker}
\opn\Am{Am} \opn\Hom{Hom} \opn\Tor{Tor} \opn\Ext{Ext}
\opn\End{End} \opn\Aut{Aut} \opn\id{id} \opn\ini{in}

\opn\nat{nat}
\opn\pff{pf}
\opn\Pf{Pf} \opn\GL{GL} \opn\SL{SL} \opn\mod{mod} \opn\ord{ord}
\opn\Gin{Gin}
\opn\Hilb{Hilb}\opn\adeg{adeg}\opn\std{std}\opn\ip{infpt}
\opn\Pol{Pol}
\opn\sat{sat}
\opn\Var{Var}
\opn\Gen{Gen}

\opn\aff{aff} \opn\con{conv} \opn\relint{relint} \opn\st{st}
\opn\lk{lk} \opn\cn{cn} \opn\core{core} \opn\vol{vol}
\opn\link{link} \opn\star{star}
\opn\gr{gr}

\def\Pc{{\mathcal P}}
\def\Qc{{\mathcal Q}}

\def\pot#1#2{#1[\kern-0.28ex[#2]\kern-0.28ex]}

\opn\dirlim{\underrightarrow{\lim}}
\opn\inivlim{\underleftarrow{\lim}}
\def\Implies{\ifmmode\Longrightarrow \else
        \unskip${}\Longrightarrow{}$\ignorespaces\fi}
\def\implies{\ifmmode\Rightarrow \else
        \unskip${}\Rightarrow{}$\ignorespaces\fi}
\def\iff{\ifmmode\Longleftrightarrow \else
        \unskip${}\Longleftrightarrow{}$\ignorespaces\fi}

\let\:=\colon
\newtheorem{Theorem}{Theorem}
\newtheorem{Lemma}[Theorem]{Lemma}

\newtheorem{Remark}[Theorem]{Remark}

\let\epsilon\varepsilon
\let\phi=\varphi
\let\kappa=\varkappa
\def\qed{\ifhmode\textqed\fi
      \ifmmode\ifinner\quad\qedsymbol\else\dispqed\fi\fi}
\def\textqed{\unskip\nobreak\penalty50
       \hskip2em\hbox{}\nobreak\hfil\qedsymbol
       \parfillskip=0pt \finalhyphendemerits=0}
\def\dispqed{\rlap{\qquad\qedsymbol}}

\opn\dis{dis}
\opn\height{height}
\opn\dist{dist}
\def\pnt{{\raise0.5mm\hbox{\large\bf.}}}

\opn\Lex{Lex}

\begin{document}
\title{Ehrhart polynomials with negative coefficients}
\author{Takayuki Hibi, Akihiro Higashitani, Akiyoshi Tsuchiya 
\\ and Koutarou Yoshida}
\thanks{
{\bf 2010 Mathematics Subject Classification:}
Primary 52B20; Secondary 52B11. \\
\, \, \, {\bf Keywords:}
integral convex polytope,
Ehrhart polynomial, $\delta$-vector.}
\address{Takayuki Hibi,
Department of Pure and Applied Mathematics,
Graduate School of Information Science and Technology,pe
Osaka University,
Toyonaka, Osaka 560-0043, Japan}
\email{hibi@math.sci.osaka-u.ac.jp}
\address{Akihiro Higashitani,
Department of Pure and Applied Mathematics,
Graduate School of Information Science and Technology,
Osaka University,
Toyonaka, Osaka 560-0043, Japan}
\email{a-higashitani@cr.math.sci.osaka-u.ac.jp}
\address{Akiyoshi Tsuchiya,
Department of Pure and Applied Mathematics,
Graduate School of Information Science and Technology,
Osaka University,
Toyonaka, Osaka 560-0043, Japan}
\email{u619884k@ecs.osaka-u.ac.jp}
\address{Koutarou Yoshida,
Department of Pure and Applied Mathematics,
Graduate School of Information Science and Technology,
Osaka University,
Toyonaka, Osaka 560-0043, Japan}
\email{u912376b@ecs.osaka-u.ac.jp}
\begin{abstract}
It is shown that, for each $d \geq 4$, there exists
an integral convex polytope $\Pc$ 
of dimension $d$ such that each of the coefficients of 
$n, n^{2}, \ldots, n^{d-2}$ of its Ehrhart polynomial
$i(\Pc,n)$ is negative.  
\end{abstract}
\maketitle
In his talk of the Clifford Lectures at Tulane University, 25--27 March 2010,
Richard Stanley gave an Ehrhart polynomial with
a negative coefficient.  More precisely,
the polynomial $\frac{\,13\,}{6}n^{3} + n^{2} - \frac{\,1\,}{6}n + 1$
is the Ehrhart polynomial of
the tetrahedron in $\RR^{3}$
with vertices $(0,0,0), (1,0,0), (0,1,0)$ and $(1,1,13)$.
See \cite[Example 3.22]{BeckRobins}.
His talk naturally inspired us to find
integral convex polytopes
of dimension $\geq 4$ whose Ehrhart polynomials possess negative coefficients.
Consult 
\cite[Part II]{HibiRedBook} and \cite[pp.~235--241]{StanleyEC}
for fundamental materials on Ehrhart polynomials.

A convex polytope is called {\em integral} if any of its vertices has
integer coordinates.  
Let $\Pc \subset \RR^N$ be an integral convex polytope of dimension $d$
and $\partial \Pc$ the boundary of $\Pc$. 
We introduce the function $i(\Pc,n)$ by setting
\[
i(\Pc,n) = \sharp(n\Pc \cap \ZZ^N), 
\, \, \, \, \, 
\text{for} 
 \, \, \, \, \, n = 1, 2, \ldots,
\] 
where $n\Pc = \{ \, n \alpha \, : \, \alpha \in \Pc \, \}$ 
and where $\sharp(X)$ is the cardinality of a finite set $X$.
The study on $i(\Pc,n)$ originated in Ehrhart \cite{Ehrhart} who showed that 
$i(\Pc,n)$ is a polynomial in $n$ of degree $d$ with $i(\Pc,0) = 1$.
Furthermore, the coefficients of $n^{d}$ and $n^{d-1}$ of $i(\Pc,n)$
are always positive (\cite[Corollary 3.20 and Theorem 5.6]{BeckRobins}).
We say that $i(\Pc,n)$ is the {\em Ehrhart polynomial} of $\Pc$. 

 

The purpose of the present paper is, for each $d \geq 4$, 
to show the existence of an integral
convex polytope of dimension $d$
such that each of the coefficients of 
$n, n^{2}, \ldots, n^{d-2}$ of its Ehrhart polynomial
$i(\Pc,n)$ is negative.  In fact,

\begin{Theorem}\label{main}
Given an arbitrary integer $d \geq 4$, 
there exists an integral convex polytope $\Pc$ of dimension $d$ 
such that 
each of the coefficients of $n$, $n^2, \ldots, n^{d-2}$ of 
the Ehrhart polynomial $i(\Pc,n)$ of $\Pc$ is negative. 
\end{Theorem}

Our proof of Theorem $1$ will be given after preparing Lemmata $2$ and $3$.

\begin{Lemma}
\label{lem1}
Let $\Pc \subset \RR^N$ be an integral convex polytope of dimension $d$ and 
$i(\Pc,n)$ its Ehrhart polynomial. Then, given an arbitrary integer $k > 0$,
there exists an integral convex polytope 
$\Pc'_{k} \subset \RR^{N+1}$ of dimension $d+1$ whose Ehrhart polynomial 
is equal to 
$(kn+1)i(\Pc,n)$.  
\end{Lemma}

\begin{proof}
It follows immediately that the Ehrhart polynomial $i(\Pc'_{k},n)$ of 
the integral convex polytope $\Pc'_{k}=\Pc \times [0,k] \subset \RR^{N+1}$
coincides with $(kn+1)i(\Pc,n)$.
\end{proof}

\begin{Lemma}
\label{lem2}
Let $d$ and $j$ be integers with $d \geq 5$ and $3 \leq j \leq d-2$, and
\[
g(d,j)=(d-3)^2 \binom{d-3}{j-1}-\binom{d-3}{j-3}.
\] 
Then one has $g(d,j) > 0$. 
\end{Lemma}
\begin{proof}
Since $d \geq 5$, one has  
$g(d,3)=(d-3)^2 \binom{d-3}{2}-1>0$ and
$g(d,d-2)=(d-3)^2-\binom{d-3}{2}>0$. 
Thus $g(d,j) > 0$
for $j = 3$ and $j = d - 2$.
Especially the assertion is true for $d = 5$ and $d = 6$. 

We now work with induction on $d$. 
Let $d \geq 7$ and $4 \leq j \leq d-3$.  Then
\begin{align*}
g(d,j)&=((d-4)^2+2d-7)\left(\binom{d-4}{j-1}+\binom{d-4}{j-2}\right)
-\left(\binom{d-4}{j-3}+\binom{d-4}{j-4}\right)\\
&=g(d-1,j)+g(d-1,j-1)+(2d-7)\binom{d-3}{j-1}. 
\end{align*}
It follows from the assumption of induction that
$g(d-1,j)+g(d-1,j-1)>0$.
Hence $g(d,j)>0$, as desired. 
\end{proof}

\begin{proof}[Proof of Theorem \ref{main}]
It is known \cite[Example 3.22]{BeckRobins} that,
given an arbitrary integer $m \geq 1$, 
there exists an integral convex polytope $\Qc_m$ of dimension 3 with 
\[
i(\Qc_m, n)=\frac{\,m\,}{6}n^3+n^2+\frac{\,-m+12\,}{6}n+1.
\]
Given an arbitrary integer $d \geq 4$,
applying Lemma \ref{lem1} with $k = d - 3$ repeatedly yields
an integral convex polytope $\Pc_m^{(d)}$ of dimension $d$ such that 
\begin{align*}
i(\Pc_m^{(d)},n) & = ((d-3)n+1)^{d-3}i(\Qc_m,n) \\
& = ((d-3)n+1)^{d-3}\left(\frac{\,m\,}{6}n^3+n^2+\frac{\,-m+12\,}{6}n+1\right). 
\end{align*}
Let $i(\Pc_m^{(d)},n) = \sum_{i=0}^d c_i^{(d,m)}n^i$ with each $c_i^{(d,m)} \in \QQ$. Then 
\[
c_1^{(d,m)}=\frac{\,-m+12\,}{6}+A_1, \quad\quad
c_2^{(d,m)}=1+\frac{\,-m+12\,}{6}A_1+A_2
\]
and
\[
c_j^{(d,m)}=\frac{\,m\,}{6}A_{j-3}+A_{j-2}+\frac{\,-m+12\,}{6}A_{j-1}+A_j,
\quad\quad
3 \leq j \leq d-2, 
\]
where 
\[
A_i=(d-3)^i\binom{d-3}{i}, \quad\quad 0 \leq i \leq d-2. 
\]

Now, since each $A_j$ is independent of $m$, 
it follows that each of $c_1^{(d,m)}$ and $c_2^{(d,m)}$ is
negative for $m$ sufficiently large.
Let $3 \leq j \leq d-2$.  One has 
\begin{eqnarray*}
c_j^{(d,m)}&=&-\frac{\,A_{j-1}-A_{j-3}\,}{6}m+(A_{j-2}+2A_{j-1}+A_j) \\
&=&-(d-3)^{j-3}\frac{\,g(d,j)\,}{6}m+(A_{j-2}+2A_{j-1}+A_j),
\end{eqnarray*}
where $g(d,j)$ is the same function as in Lemma \ref{lem2}.
Since $g(d,j)>0$, it follows that 
$c_j^{(d,m)}$ can be negative for $m$ sufficiently large. 
Hence, for $m$ sufficiently large, the integral convex polytope 
$\Pc_m^{(d)}$ of dimension $d$ enjoys the required property.
\end{proof}

We conclude this paper with

\begin{Remark}
{\em
The polynomial
\begin{align*}
i(\Qc_m, n) &= \frac{\,m\,}{6}n^3+n^2+\frac{\,-m+12\,}{6}n+1 \\
&= \frac{\,1\,}{6}(n + 1)(m n^{2} + (6 - m) n + 6)
\end{align*}
has a real positive zero for $m$ sufficient large.  Hence
$i(\Pc_m^{(d)},n)$ has a real positive zero for $m$ sufficient large. 

Thus in particular, for $m$ sufficient large and for an arbitrary integral
convex polytope $\Qc$, the Ehrhart polynomial
$i(\Pc_m^{(d)} \times \Qc,n)$ of $\Pc_m^{(d)} \times \Qc$ 
also possesses a negative coefficient.
}
\end{Remark} 

We are grateful to Richard Stanley for his suggestion on real positive roots of
Ehrhart polynomials.

\end{document}